\documentclass[12pt,a4paper,reqno]{amsproc}
\usepackage{amsmath, amsthm, amscd, amsfonts, amssymb, graphicx, color}
\usepackage{subcaption}
\usepackage[bookmarksnumbered, colorlinks, plainpages]{hyperref}
\hypersetup{colorlinks=true,linkcolor=red, anchorcolor=green, citecolor=cyan, urlcolor=red, filecolor=magenta, pdftoolbar=true}

\newtheorem{theorem}{Theorem}[section]
\newtheorem{lemma}[theorem]{Lemma}

\theoremstyle{definition}
\newtheorem{definition}[theorem]{Definition}

\theoremstyle{remark}
\newtheorem{remark}[theorem]{Remark}
\newtheorem{note}[theorem]{Note}

\numberwithin{equation}{section}

\begin{document}

\setcounter{page}{1}

\title[Uncertainty principles associated with LCDT]
 {Uncertainty principles associated with the  linear canonical Dunkl transform}
 
\author[Umamaheswari S and Sandeep Kumar Verma]{Umamaheswari S and Sandeep Kumar Verma}

\address{Department of Mathematics, SRM University-AP, Andhra Pradesh, Guntur--522502, India}

\email{umasmaheswari98@gmail.com, sandeep16.iitism@gmail.com}



\maketitle
\begin{abstract} 
In this paper, we establish analogs of Miyachi, Cowling-Price, and Heisenberg-Pauli-Weyl uncertainty principles in the framework of the linear canonical Dunkl transform. We also obtain some weighted inequalities, such as Nash, Clarkson, Donoho-Stark, and Matolcsi-Szuc’s type inequalities.
\end{abstract}
\section{Introduction}
The classical Heisenberg-Weyl uncertainty principle provides insights into the relationship between a function and its Fourier transform. It states that attempting to confine the behavior of one leads to a loss of control over the other. A key implication of Heisenberg-Weyl uncertainty principle in physics, particularly in optics, is that it sets a lower bound for the product of a signal's spread and its bandwidth. For instance, it means that the product of the effective widths of light intensity in both the space and frequency domains has a minimum value, which is achieved when the light is completely coherent and Gaussian. The Heisenberg-Weyl uncertainty relation is frequently used in signal processing to examine localization in time-frequency analysis \cite{almeida1994fractional}. The classical formulation of the uncertainty principle \cite{folland1997uncertainty} in the form of the lower bound of the product of the dispersions of a function $f$ and its Fourier transform $\hat{f}$ is as follows
\begin{eqnarray}  \label{e:1.1} \||x|f\|_{L^2(\mathbb{R}^n)}\||\omega|\hat{f}\|_{L^2(\mathbb{R}^n)}\geq C\|f\|^2_{L^2(\mathbb{R}^n)}.
\end{eqnarray}
There are many variations of the uncertainty principle. Broadly uncertainty principles can be categorized into two main types: quantitative and qualitative. 
Quantitative uncertainty principles \cite{mejjaoli2022new, Sahbani} are precise inequalities that provide valuable insights into the relationship between a function and its Fourier transform. These inequalities, reminiscent of the classical Heisenberg uncertainty principle, detail how the temporal concentration of a function corresponds to its spectral dispersion. This includes Heisenberg-Pauli-Weyl, Nash-type \cite{carlen1993sharp}, Clarkson-type \cite{Carlson}, Donoho-Stark \cite{donoho1989uncertainty}, and Matlocsi-Szucs-type uncertainty principles \cite{matolcsi1973intersection}, etc.
\par On the other hand, the qualitative uncertainty principles \cite{folland1997uncertainty}, exemplified by works such as those by Hardy \cite{hardy1933theorem, sitaram1997}, Cowling-Price \cite{cowling1983generalisations}, and Miyachi \cite{miyachi1997generalization}, are expressed as theorems rather than inequalities. They describe the behavior of a function and its Fourier transform under specific conditions.
\par The various uncertainty principles in several integral transform domains, such as the LCT transform \cite{guanlei2010uncertainty}, fractional Fourier transform \cite{shinde2001uncertainty}, the Hankel transform \cite{rosler1999uncertaintyy},  Dunkl transform \cite{ghobber2013uncertainty, rosler1999uncertainty, soltani2013heisenberg, soltani2013general, soltani2014p, soltani2017uncertainty}, Heckman-Opdam transform \cite{johansen2016uncertainty}, linear canonical deformed Hankel transform \cite{mejjaoli2023linear}, etc., have been established in great detail.

\par { The main goal of this article is to establish various quantitative and qualitative uncertainty principles within the framework of the linear canonical Dunkl transform on the real line. }
{The classical linear canonical transform (LCT) (or linear canonical Fourier transform)  was introduced independently by Collins \cite{Collins} in the context of paraxial optics and by Moshinsky \cite{Moshinsky} in quantum mechanics to understand the conservation of information and uncertainty under linear maps of phase space.
The linear canonical transform (LCT) encompasses numerous well-known transforms in optics, including the Fourier transform \cite{almeida1993introduction}, fractional Fourier transform \cite{alieva1994fractional, alieva1999mode, namias1980fractional}, and Fresnel transform \cite{james1996generalized}, etc. The LCT offers greater flexibility compared to other transforms due to its additional degrees of freedom. This characteristic has sparked significant interest in its application to address challenges encountered in optics, quantum physics, and signal processing domains \cite{ballentine1970statistical, cramer1986transactional, eskov2014uncertainty, faris1978inequalities, hilgevoord1996uncertainty, uffink1985uncertainty}.
In addition to its wide-ranging applications, the theoretical development of the linear canonical transform (LCT) has been extensively studied and investigated, including convolution structure, sampling theorems, uncertainty principles, and other theoretical aspects. Further, the linear canonical transform associated with the various integral transforms have been defined and studied, for instance, the linear canonical Hankel transform \cite{prasad, srivastava2021framework}, the linear canonical deformed Hankel transform \cite{mejjaoli2023linear}, the linear canonical Fourier Bessel transform \cite{dhaouadi2021harmonic}, etc. }

To proceed further, let us recall the definition of the linear canonical Dunkl transform introduced by Ghazouani et al. \cite{ghazouani2017unified} in connection with the Dunkl transform.  
\begin{definition} \label{D:1}
Let $f$ be an integrable function on $\mathbb{R}$ and $M=(a,b;c,d)\in SL(2,\mathbb{R})$.  Then the  linear canonical Dunkl transform is defined as \cite{ghazouani2017unified} 
\begin{equation*} 
 D_k^M(f)(\lambda) = \left \{\begin{array}{ll} \frac{1}{(ib)^{k +1}}  \int_{\mathbb{R}}  f(x)\, E_k^M(\lambda,x)\,d\mu_k(x), \quad & b \neq 0, 
 \\
 \frac{e^{i\frac{c}{2a}\lambda^2}}{|a|^{k+1}}\,f(\lambda/a) & b=0
 \end{array}\right.
\end{equation*}
where the kernel $E^M_{k}(\lambda,x)$ is defined by 
\begin{equation*}
E^M_{k}(\lambda,x) = e^{\frac{i}{2}(\frac{d}{b}\lambda^2+\frac{a}{b}x^2)}E_{k}(-i\lambda/b,x). 
\end{equation*}
\end{definition}
It is easy to verify that the linear canonical Dunkl transform (LCDT) with parameters $(\cos(\theta), -\sin(\theta); \sin(\theta), \cos(\theta))$ reduces to the fractional Dunkl transform \cite{gh}. In the specific case of $\theta=\pi/2$, it becomes the Dunkl transform, and when $\theta=\pi/2$ and $k=0$, it becomes Fourier transform.

We will precisely establish the following uncertainty principles associated with the linear canonical Dunkl transform:

\begin{itemize}
\item \textbf{Heisenberg-Pauli-Weyl uncertainty principle:} Let $f\in L^p_k(\mathbb{R})$,\, $1<p\le2$, $0<\alpha<\frac{2(k+1)}{q}$, and $\beta>0$. Then there exists a constant $C>0$ such that 
\begin{equation*}
    \|D_k^M(f)\|_{L^{q}_k(\mathbb{R})} \le C\, \||y|^\alpha\,f\|^{\frac{\beta}{\alpha+\beta}}_{L_k^p(\mathbb{R})}\, \||\lambda|^\beta\, D_k^M(f)\|^{\frac{\alpha}{\alpha+\beta}}_{L_k^{q}(\mathbb{R})}. 
\end{equation*}
\item \textbf{Nash-type inequality:} Let $1<p_1<p_2\le2$ with $q_1$ and $q_2$ being the conjugate exponents of $p_1$ and $p_2$, respectively. Suppose $f$ belongs to both $ L_k^{p_1}(\mathbb{R})$ and $ L_k^{p_2}(\mathbb{R})$. Then there exists a constant $C(q_1,q_2,k,b,s)$ such that

\begin{eqnarray*}
\|D_k^M(f)\|_{L_k^{q_2}(\mathbb{R})} &\le&  C(q_1,q_2,k,b)\, \||x|^s\, D_k^M(f)\|_{L_k^{q_2}(\mathbb{R})}^{\frac{(2k+2)\,(q_1-q_2)}{(2k+2)\,(q_1-q_2)+sq_1q_2}}\\ &&\times \, \|f\|^{\frac{sq_1q_2}{(2k+2)\,(q_1-q_2)+sq_1q_2}}_{L_k^{p_1}(\mathbb{R})}, \quad\text{for some }\,\, s>0,
\end{eqnarray*}
 and
\begin{equation*}
 C(q_1,q_2,k,b) = \left\{ 1+\left(\frac{1}{2^{k+1}\,\Gamma(k+2) }\right)^{\frac{q_1-q_2}{q_1}}\, C_{k,b}^{(1-\frac{2}{q_1})q_2}\right\} ^{\frac{1}{q_{2}}}.  
\end{equation*}
\item {\textbf{Clarkson-type inequality:}} If $f$ belongs to $ L_k^{p_1}(\mathbb{R})$ and $ L_k^{p_2}(\mathbb{R}) $, where $1<p_1<p_2\le 2$, then there exists a constant $C(p_1,p_2,k)$ such that 
\begin{equation*}
    \|f\|_{L_k^{p_1}(\mathbb{R})} \le C(p_1,p_2,k)\, \||y|^s\,f\|^{\frac{(2k+2)\,(p_2-p_1)}{(2k+2)(p_2-p_1)+p_1p_2s}}_{L_k^{p_1}(\mathbb{R})}\, \|f\|_{L_k^{p_2}(\mathbb{R})}^{\frac{p_1p_2s}{(2k+2)(p_2-p_1)+p_1p_2s}},
\end{equation*}
where
\begin{equation*}
  C(p_1,p_2,k) = \left\{ 1+\frac{1}{(2^{k+1}\,\Gamma(k+2))^{\frac{p_2-p_1}{p_2}}}\right\}^{\frac{1}{p_1}}.  
\end{equation*}
\item {\textbf{Donoho-Stark uncertainty principle:}}  Let $E$ and $F$ be measurable subsets of $\mathbb{R}$, and let $ f\in L_k^{p_1}(\mathbb{R})\cap L^{p_2}_k(\mathbb{R})$ for $1<p_1<p_2\le 2$. If $f$ is $\epsilon_E$-concentrated in $L_k^{p_1}(\mathbb{R})$ and $D_k^M(f)$ is $\epsilon_{F}$ concentrated in $L_k^{q_2}(\mathbb{R})$, then 
  \begin{equation*}
      \|D_k^M(f)\|_{L_k^{q_2}(\mathbb{R})} \le C_{k,b}^{1-\frac{2}{q_1}}\, \frac{(\gamma_k(E))^{\frac{p_2-p_1}{p_1\,p_2}}\, (\gamma_k(F))^{\frac{q_1-q_2}{q_1\,q_2}}}{(1-\epsilon_E)\,(1-\epsilon_{F})}\, \|f\|_{L_k^{p_2}(\mathbb{R})}.
  \end{equation*}
  \item {\textbf{Matolcsi-Szucs type uncertainty principle:}} If the function  $f$ belongs to $ L_k^{p_1}(\mathbb{R})$ and $L_k^{p_2}(\mathbb{R})$,  where $1<p_1\le p_2 \le 2$, then 
\begin{equation*}
     \|D_k^M(f)\|_{L_k^{q_2}(\mathbb{R})} \le C_{k,b}^{1-\frac{2}{q_1}}\, (\gamma_k(A_{D_k^M(f)}))^{\frac{q_1-q_2}{q_1\,q_2}}\, (\gamma_k(A_f))^{\frac{p_1-p_2}{p_1\,p_2}}\, \|f\|_{L^{p_2}_k(\mathbb{R})},
\end{equation*}
where $$A_{D_k^M(f)}= \left\{x\in \mathbb{R}: D_k^M(f)(x) \neq 0\right\}$$ and $$A_f = \left\{ t
\in \mathbb{R}:f(t)\neq0\right\}.$$

\item \textbf{Miyachi's uncertainty principle:} Let $f$ be a measurable function on $\mathbb{R}$ such that $e^{sx^2}f\in L^p_k(\mathbb{R})+L^{q}_k(\mathbb{R}),$ where $  p,q\in [1, \infty]$. Suppose that
\begin{equation*}
\int_{\mathbb{R}} ln^+\left( \frac{|e^{tx^2} D_k^M(f)(x)|}{\lambda}\right) dx< \infty,    
\end{equation*}
for some positive real constants $s,t,\lambda$, and 
\begin{equation*}
ln^+(r)=\left \{\begin{array}{ll}
ln(r)& 
\text{if}\,\, r>1\\
0, & \text{otherwise}.
\end{array}\right. 
\end{equation*} 
Then
\begin{itemize}
    \item [$(i)$]  $f=0$  a.e on $\mathbb{R}$, if $st>\frac{1}{4b^2}$.
    \item[$(ii)$] $f(x) =  C\, e^{-\left(\frac{i}{2}\frac{a}{b}+s\right)x^2}$, if $st= \frac{1}{4b^2}$ and $|C|\le |\lambda|$.
    \item[$(iii)$] There are many functions satisfying the hypothesis, if $st<\frac{1}{4b^2}$.
\end{itemize}
\vspace{5pt}
\item {\textbf{Cowling-Price uncertainty principle:}} Let $f$ be a measurable function on $\mathbb{R}$  such that
\begin{equation*}
\int_{\mathbb{R}} \frac{e^{spx^2}\,|f(x)|^p}{(1+|x|)^n}\, d\mu_k(x) <\infty 
\end{equation*}
and 
\begin{equation*} 
\int_{\mathbb{R}} \frac{e^{tq\lambda^2} |D_k^M(f)(\lambda)|^q}{(1+|\lambda|)^m}\, d\lambda <\infty,
\end{equation*}
where $s,t>0$ and $n>0, m>1, 1\le p,q <\infty$. Then the following results hold.
\begin{itemize}
    \item [$(i)$] If $st>\frac{1}{4b^2}$, then $f=0$ a.e on $\mathbb{R}$.
    \item[$(ii)$] If $st= \frac{1}{4b^2}$, then  $f(x) = Q(x)\, e^{-(s+\frac{i}{2}\frac{a}{b})x^2}$, where $Q$ is a polynomial with $deg(Q) \le min \{\frac{n}{p}+ \frac{2k}{p'}, \frac{m-1}{q}\}$ and $p'$ is the conjugate exponent of $p$. Moreover, if  $1<m \le 1+q$ and $n>2k+1$, then $f(x)= C\,  e^{-(s+\frac{i}{2}\frac{a}{b})x^2}$. 
\item[$(iii)$] If $st<\frac{1}{4b^2}$, then there are many functions of the form $f(x) = Q(x)\,e^{-(\delta
+\frac{i}{2}\frac{a}{b})x^2}$ which satisfy the hypothesis, where $\delta \in (t, \frac{1}{4sb^2})$ and $Q$ is the polynomial in $\mathbb{R}$. 
\end{itemize}
\end{itemize}

This paper is organized as follows:
Section \ref{S:2} contains the fundamental properties of the linear canonical Dunkl transform. Section \ref{s:5} focuses on deriving the Heisenberg-Pauli-Weyl uncertainty principle and various Heisenberg-type inequalities, such as Nash-type and Clarkson-type inequalities. Further, we discuss the Donoho-Stark and Matolcsi-Szuc’s uncertainty principles. Section \ref{S:3} establishes the Miyachi and Cowling-Price uncertainty principle for the linear canonical Dunkl transform. 
\begin{section}{Preliminaries} \label{S:2}
We begin this section by introducing some notations. We denote $\mathcal{C}_0 (\mathbb{R})$ the space of continuous functions on $\mathbb{R}$ which vanishes at infinity and  $ M = \begin{bmatrix} a & b\\ c & d\\ \end{bmatrix}$ be an arbitrary matrix in  the special linear group $SL(2,\mathbb{R})$. For the notational convention, we will write $M=(a,b;c,d)$.
Throughout this article, we will consider $k\ge - \frac{1}{2}$. For $1\leq p\leq\infty$, we denote by $ L_k^p(\mathbb{R})$ the Banach   space consisting of measurable functions $f$ on $\mathbb{R}$ equipped with the norm 
   \begin{eqnarray*}
    \|f\|_{L_k^p(\mathbb{R})} &=&\left( \int_{\mathbb{R}} |f(x)|^p\,  d\mu_k(x)\right)^{\frac{1}{p}},~ \text{if}~ 1\le p<\infty, 
\end{eqnarray*}
and 
\begin{eqnarray*}
   \|f\|_{L_k^{\infty}(\mathbb{R})} &=& \mathop{\underset{{x\in \mathbb{R}}}{\text{ess.sup}}}|f(x)|, \quad p= \infty, 
\end{eqnarray*}
where $d\mu_k(x) = (|x|^{2k+1}/2^{k+1}\Gamma(k+1))dx.$ 
 
\begin{subsection}{ Properties of Linear 
canonical Dunkl transform}
In this subsection, we recall some properties of the linear canonical Dunkl transform from \cite{ghazouani2017unified} that will be useful for subsequent investigations.

\begin{itemize}
\item[$(i)$]  Inversion formula: For every $ f\in L^1_{k}(\mathbb{R})$ such that $D^M_{k}f \in L^1_{k}(\mathbb{R})$,  we have
\begin{equation*}\label{e:2.2}
D^{M^{-1}}_{k} (D^M_{k}f)
=f ,\,\, \text{a.e},    
\end{equation*}
where $M^{-1}$ is the inverse of matrix $M$.
\item[$(ii)$] Plancherel's formula:
If $f \in  L^{1}_{k}(\mathbb{R})\cap L^{2}_{k}(\mathbb{R}) $, then  $D^M_k(f) \in L^{2}_{k}(\mathbb{R})$ and we have 
 \begin{equation} \label{e:2.3}
 \|D^M_k(f)\|_{ L^{2}_{k}(\mathbb{R})} = \| f \|_{ L^{2}_{k}(\mathbb{R})}.
 \end{equation}
\end{itemize}
\begin{remark} 
The linear canonical Dunkl transform and the Dunkl transform are related as follows
\begin{eqnarray} \label{eq:4.2}
 D_k^M(f)(\lambda) 
 &=& \frac{e^{\frac{i}{2}\frac{d}{b}\lambda^2}}{(ib)^{k+1}}\, D_k(\tilde{f})\left(\frac{\lambda}{b}\right),~b\neq 0,
\end{eqnarray}
where $\tilde{f}(x) =  e^{\frac{i}{2}\frac{a}{b}x^2}f(x)$ and $D_k(f)$ is the Dunkl transform \cite{de1993dunkl, rosler1998generalized} defined by 
\begin{equation*}
  D_k(f)(\lambda) = \int_{\mathbb{R}} f(x)\, E_k(-i\lambda,x)\,d\mu_k(x),
\end{equation*}
where  $ E_k(-i\lambda,x)$ is the Dunkl kernel \cite{rosler2003positive, rosler2003dunkl}
defined by
 \begin{equation*}
  E_k(-i\lambda,x) = j_{k}(-\lambda x)+\frac{(-i\lambda) x}{2(k+1)}j_{k+1}(-\lambda x), \quad k\ge -1/2,\end{equation*}
and $j_{k}$  denotes the normalized spherical Bessel function 
 \begin{equation*}
 j_{k}(x) = 2^{k} \Gamma(k+1)x^{-k}J_{k}(x) = \Gamma(k +1)\sum^{\infty}_{n=0}\frac{(-1)^n(x/2)^{2n}}{n!\Gamma(n+k+1)}.
\end{equation*}
\end{remark} 
For $\lambda,x\in\mathbb{R}$, the Dunkl kernel $E_k(-i\lambda, x)$ is the unique solution of the system
\begin{eqnarray*}
    \Lambda_k f(x)=-i\lambda f(x) ~\text{and} ~f(0)=1,
\end{eqnarray*}
where the Dunkl operator $\Lambda_k$ associated with the reflection group $\mathbb{Z}_2$ on $\mathbb{R}$ is defined as
\begin{eqnarray*}
    \Lambda_k(f)(x)=\frac{d}{dx}f(x)+\frac{2k+1}{2}\left(\frac{f(x)-f(-x)}{x}\right), ~k\geq -1/2.
\end{eqnarray*}
Note that $\Lambda_{-1/2}=d/dx$.
\\
Let us collect some properties of $E_k(-i\lambda,x)$ from \cite{rosler2003positive,  trimeche2001dunkel}:

\subsection{Properties:}\label{pr:2.2}
\begin{itemize}
\item[$(i)$] The kernel $E_k(-i\lambda, x)$ extends as an analytic function of $\mathbb{C}\times \mathbb{C}.$ 
\item[$(ii$)] For all $x, \lambda \in \mathbb{R}$, we have $|E_k(-i\lambda,x)|\le 1$.
\item[$(iii)$] For all $x\in \mathbb{R}$ and $z\in \mathbb{C}$, we have $\left|E_k\left(-iz,x\right)\right| \le e^{|\mathfrak{Im}z||x|}$.
\item[$(iv)$] For $s>0$, we have $$\int_\mathbb{R} e^{-s\xi^2}\, E_k(-ix, \xi)\, E_k(iy, \xi) d\mu_k(\xi) = \frac{\Gamma(k+1)}{s^{k+1}}\,e^{-\frac{(x^2+y^2)}{4s}}\,E_k\left( \frac{x}{\sqrt{2s}}, \frac{y}{\sqrt{2s}}\right).$$
\end{itemize}
\begin{note} \label{N:2.4}
   In \cite{kawazoe2010uncertainty}, it is proved that $D_k(f)$ can be holomorphically extended to an entire function in $\mathbb{C}$. As a consequence, the LCDT $D_k^M(f)$ can also be extended holomorphically to $\mathbb{C}$.    
\end{note}
\begin{lemma}[Riemann-Lebesgue lemma] \label{le:3.1} If $f\in L_k^1(\mathbb{R})$, then  $D_k^M(f)$ belongs to $\mathcal{C}_0(\mathbb{R})$ and satisfies the following inequality 
\begin{equation*} \label{e:2.4}
    \|D_k^M(f)\|_{L_k^\infty(\mathbb{R})} \le \frac{1}{|b|^{k+1}}\, \|f\|_{L_k^1(\mathbb{R})}.
\end{equation*}
\end{lemma}
Now, we will establish Young's inequality in the following lemma.
\begin{lemma}[Young's inequality] \label{le:3.2} For $ 1 \le p \le 2$, the operator $D_k^M$ is a bounded linear operator on $L_k^p(\mathbb{R})$, and it satisfies the following inequality:   \begin{equation}\label{eq:3.1}
    \|D_k^M(f)\|_{L_k^{q}(\mathbb{R})} \le \frac{1}{|b|^{(k+1)\left(1-\frac{2}{q}\right)}}\, \|f\|_{L^p_k(\mathbb{R})},
\end{equation}
where $q$ is the conjugate exponent of $p$.
\end{lemma} 
\begin{proof}
By applying the Riesz-Thorin interpolation theorem \cite{stein1956interpolation} between the Plancherel formula \eqref{e:2.3} and the Riemann-Lebesgue Lemma \ref{le:3.1}, we obtain the desired result.
\end{proof}
\end{subsection}
\end{section}

\section{Quantitative Uncertainty Principles}\label{s:5}
The quantitative uncertainty principle provides precise inequalities that relate a function and its transform. It offers exact mathematical bounds and relations. These precise bounds are useful in practical applications, such as signal processing, quantum mechanics, and differential equations. In this section, we will prove the Heisenberg-Pauli-Weyl uncertainty principle, Nash and Clarkson type uncertainty principles, and Donoho-Stark and Matolcsi-Szucs uncertainty principles in the setting of the linear canonical Dunkl transform (LCDT).
\begin{subsection}{Heisenberg's-Pauli-Weyl Uncertainty Principle}
The Heisenberg-Pauli-Weyl uncertainty principle \cite{soltani2013heisenberg}  can be established through various methods such as the entropy-based approach and the spectral method. However, in this study, we focus on proving the Heisenberg-Pauli-Weyl uncertainty principle in the classical approach which generalizes equation \eqref{e:1.1}.
\begin{theorem}\label{t:4.3}
Let $f\in L^p_k(\mathbb{R})$, \,$1<p\le2$, $0<\alpha<\frac{2(k+1)}{q}$ and $\beta>0$. Then there exist $C>0$  such that 
\begin{equation*}
\|D_k^M(f)\|_{L^{q}_k(\mathbb{R})} \le C\, \||y|^\alpha\,f\|^{\frac{\beta}{\alpha+\beta}}_{L_k^p(\mathbb{R})}\, \||\lambda|^\beta\, D_k^M(f)\|^{\frac{\alpha}{\alpha+\beta}}_{L_k^{q}(\mathbb{R})} .
\end{equation*}
\end{theorem}
In order to simplify the proof of Theorem \ref{t:4.3}, we first establish the following technical lemmas:
\begin{lemma}\label{l:5.1}
    Let $t>0$ and  $p\in [1,\infty)$. Then there exists $C>0$ such that
    \begin{eqnarray*}
\|e^{-t\lambda^2}\|_{L^p_k(\mathbb{R})} =C\, t^{-\left(\frac{k+1}{p}\right)}.
    \end{eqnarray*}
\end{lemma} 
\begin{proof}
    The proof directly follows from the norm estimate.
\end{proof}
\begin{lemma}\label{l:6.2}
    If  $f\in L_k^p(\mathbb{R})$, where $1<p\le2$, and $0<\alpha< \frac{2(k+1)}{q}$, with $q$ being the conjugate exponent of $p$, then there exists  $C(\alpha,b,q)>0$ such that 
    \begin{equation*}
\|e^{-t\lambda ^2}  D_k^M(f)\|_{{L_k^{q}(\mathbb{R})}} \le C(\alpha,b,q)\, t^{-\frac{\alpha}{2}}\, \| |y|^\alpha f\|_{{L^p_k(\mathbb{R})} }, \,\, \text{for some} \,\, t>0.    
    \end{equation*}
\end{lemma}
\begin{proof}
The result directly holds, if $\||y|^\alpha\,f\|_{L_k^p(\mathbb{R})} =\infty$. Now, let us assume that $\||y|^{\alpha}\, f\|_{L_k^p(\mathbb{R})}<\infty.$
Let $f\in L^p_k(\mathbb{R})$, and suppose $f$ can be written as 
\begin{equation} \label{eq :5.1}
f = f \chi_{(-r,r)}+f \chi_{\mathbb{R}\setminus(-r,r)}, \quad r>0,
\end{equation}
where $\chi_{(-r,r)}$ is the characteristic function on $(-r,r)$.
Also, we can easily check that
\begin{equation}\label{5.1}
 |f \chi_{{\mathbb{R}}\setminus(-r,r)}(y)|\le |y|^\alpha\,r^{-\alpha}\,|f(y)|.   
\end{equation}
By applying the linear canonical Dunkl transform and multiplying both sides of equation \eqref{eq :5.1} by $e^{-t\lambda^2}$, and then applying Minkowski's inequality, we obtain
\begin{equation}\label{e:5.3}
    \|e^{-t\lambda^2}D_k^M(f)\|_{L_k^q(\mathbb{R})} \le \|e^{-t\lambda^2} D_k^M\left(f\chi_{(-r,r)}\right)\|_{L_k^{q}(\mathbb{R})}+\|e^{-t\lambda^2}D_k^M\left(f\chi_{\mathbb{R}\setminus(-r,r)}\right)\|_{L_k^{q}(\mathbb{R})}.
\end{equation}
Now, we consider the first part of equation \eqref{e:5.3}. By successively applying Lemmas \ref{l:5.1}, \ref{le:3.1}, and using H\"older's inequality, we obtain the following:
\begin{eqnarray*}
\|e^{-t\lambda^2} D_k^M\left(f\chi_{(-r,r)}\right) \|_{L_k^{q}(\mathbb{R})} &\le& \|e^{-t\lambda^2}\|_{L_k^{q}(\mathbb{R})}\, \| D_k^M\left(f\chi_{(-r,r)}\right)\|_{L_k^{\infty}(\mathbb{R})}.\\
&\le&  \frac{C\, t^{-\left(\frac{k+1}{q}\right)}}{|b|^{k+1}}\, \|f\chi_{(-r,r)}\|_{L_k^1(\mathbb{R})}\\
     &\le& \frac{C\, t^{-\left(\frac{k+1}{q}\right)}}{|b|^{k+1}}\, \||y|^{-\alpha} \,\chi_{(-r,r)}\|_{L_k^{q}(\mathbb{R})}\, \||y|^\alpha\,f\|_{L_k^{p}(\mathbb{R})}.
 \end{eqnarray*}
For $0<\alpha<\frac{2(k+1)}{q}$, we have
\begin{equation*}\label{e:5.6}
     \||y|^{-\alpha}\,\chi_{(-r,r)}\|_{L_k^{q}(\mathbb{R})} = \frac{r^{-\alpha+\frac{2(k+1)}{q}}}{(-\alpha q+2(k+1))^{\frac{1}{q}}}.
 \end{equation*}
 We deduce that
\begin{equation}\label{eq:5.6}
\|e^{-t\lambda^2} D_k^M\left(f\chi_{(-r,r)}\right) \|_{L_k^{q}(\mathbb{R})} \le  \frac{C\, t^{-\left(\frac{k+1}{q}\right)}\,r^{-\alpha+\frac{2(k+1)}{q}}}{|b|^{k+1}\,(-\alpha q+2(k+1))^{\frac{1}{q}}}\, \||y|^\alpha\,f\|_{L_k^{p}(\mathbb{R})}.   
\end{equation}
On the other hand
 \begin{equation*}
\|e^{-t\lambda^2}D_k^M\left(f\chi_{\mathbb{R}\setminus(-r,r)}\right)\|_{L_k^{q}(\mathbb{R})} \le \|e^{-t\lambda^2}\|_{L^\infty_k(\mathbb{R})}\, \|D_k^M\left(f\chi_{\mathbb{R}\setminus(-r,r)}\right)\|_{L_k^{q}(\mathbb{R})}.
 \end{equation*}
By using Young's inequality \eqref{eq:3.1} and equation \eqref{5.1}, we have
 \begin{eqnarray}\label{e:5.7}
 \|e^{-t\lambda^2}D_k^M\left(\chi_{\mathbb{R}\setminus(-r,r)}\,f\right)\|_{L_k^{q}(\mathbb{R})} &\le&   
 \frac{r^{-\alpha}}{|b|^{(k+1)\left(1-\frac{2}{q}\right)}} \, \||y|^\alpha f\|_{L^p_k(\mathbb{R})}.
 \end{eqnarray}
By taking $r = t^{\frac{1}{2}}$ and combining equations  \eqref{eq:5.6} and \eqref{e:5.7}, we obtain the required result.
\end{proof}
\begin{proof} {\bf of Theorem \ref{t:4.3}.}
    Let $f\in L_k^p(\mathbb{R})$, where $1<p\le 2$. For $t>0$, we consider
\begin{equation} \label{e:4.6}
\|(1-e^{-t\lambda^2})\, D_k^M(f)\|_{L_k^{q}(\mathbb{R})} \le t^{\frac{\beta}{2}}\, \|(t\lambda^2)^{\frac{-\beta}{2}}\, (1-e^{-t\lambda^2})\|_{L_k^\infty(\mathbb{R})}\, \||\lambda|^\beta\, D_k^M(f)\|_{L_k^{q}(\mathbb{R})}. 
\end{equation}
We will give the proof by dividing 
$\beta$ into two parts: $\beta\leq 2$ and $\beta> 2$ . \\
 If $\beta \le 2$, then $r^{\frac{-\beta}{2}}(1-e^{-r})$ is bounded for $r \ge0$. 
In view of  Lemma \ref{l:6.2} and estimate \eqref{e:4.6}, we obtain that
\begin{equation*} 
    \|D_k^M(f)\|_{L_k^{q}(\mathbb{R})}\le C\left( t^{-\frac{\alpha}{2}}\, \| |y|^\alpha f\|_{{L^p_k(\mathbb{R})} }+ t^{\frac{\beta}{2}}\||\lambda|^\beta\, D_k^M(f)\|_{L_k^{q}(\mathbb{R})}\right).
\end{equation*}
By choosing  $t = \left( \frac{\alpha}{\beta}\, \frac{\| |y|^\alpha f\|_{{L^p_k(\mathbb{R})} }}{\||\lambda|^\beta\, D_k^M(f)\|_{L_k^{q}(\mathbb{R})}}\right)^{\frac{2}{\alpha+\beta}}$, we get the required result
\begin{equation} \label{e:5.10}
    \|D_k^M(f)\|_{L^{q}_k(\mathbb{R})} \le C\, \||y|^\alpha\,f\|^{\frac{\beta}{\alpha+\beta}}_{L_k^p(\mathbb{R})}\, \||\lambda|^\beta\, D_k^M(f)\|^{\frac{\alpha}{\alpha+\beta}}_{L_k^{q}(\mathbb{R})}, \quad \text{for} \,\,\,\beta\le 2.
\end{equation}
 Now, we consider the second case $\beta >2$. Let us take $u$ to lie between $0$ and $\beta$, which gives the inequality
\begin{equation}\label{5.6}
    r^{u} \le 1+r^{\beta}, \qquad\text{for}\,\, r= \frac{|\lambda|}{\epsilon}, \,\,\epsilon >0.
\end{equation}
From \eqref{5.6}, we can obtain that
\begin{equation*}
\| |\lambda|^{u} \, D_k^M(f)\|_{L^{q}_k(\mathbb{R})} \le \epsilon^{u}\, \|D_k^M(f)\|_{L^{q}_k(\mathbb{R})}\,+\, \epsilon^{u-\beta}\, \| |\lambda|^{\beta}\, D_k^M(f)\|_{L^{q}_k(\mathbb{R})}.
\end{equation*}
By choosing 
$\epsilon = \left( \frac{\beta-u}{u}\right)^{\frac{1}{\beta}} \, 
\left(\frac{\||\lambda|^\beta\, D_k^M(f)\|^{\frac{1}{\beta}}_{L_k^{q}(\mathbb{R})}}{\|D_k^M(f)\|^{\frac{1}{\beta}}_{L_k^{q}(\mathbb{R})}}\right)$, we deduce that
\begin{equation} \label{e:5.12}
\| |\lambda|^{u} \, D_k^M(f)\|_{L^{q}_k(\mathbb{R})} \le \frac{\beta}{\beta-u}(\beta-u)^\frac{u}{\beta}\,\||\lambda|^\beta\, D_k^M(f)\|^{\frac{u}{\beta}}_{L_k^{q}(\mathbb{R})}\,\|D_k^M(f)\|^{(1-{\frac{u}{\beta})}}_{L_k^{q}(\mathbb{R})}. 
\end{equation}
It is evident that from equations \eqref{e:5.10} and \eqref{e:5.12}, we achieve the required result.
\end{proof}

\begin{subsection}{Nash and Clarkson type uncertainty principles}
 The classical Nash inequality, introduced by Nash \cite{nash1958continuity} expressed as
\begin{equation}\label{e:4.11}
    \|f\|_2^{2+\frac{4}{n}}\le C_n\,\|f\|^{\frac{4}{n}}_{1}\, \||\xi| \hat{f}\|^2_2,
\end{equation}
where  $f\in L^1(\mathbb{R}^n)\cap L^2(\mathbb{R}^n)$.  It demonstrates how quantum mechanics imposes limits on the simultaneous measurement of time and energy. This inequality was developed to establish regularity properties for solutions to parabolic partial differential equations, and the optimal constant
 $C_n$ was computed in  \cite{carlen1993sharp}. Analogously to equation \eqref{e:4.11}, we will establish a Nash-type inequality for the LCDT in the following theorems. For the sake of notation, we define $C_{k,b} = \frac{1}{|b|^{k+1}}$, and throughout this section, $q$ represents the conjugate exponent of $p$.

 \begin{theorem}\label{Th:3.4}
Let  $f$ belongs to $ L
 _k^1(\mathbb{R})$ and $L^p_k(\mathbb{R}), \,\, 1<p\le 2$. Then 
 \begin{equation*}
 \|D_k^M(f)\|_{L_k^{q}(\mathbb{R})} \le C(b,q,k) \, \|f\|_{L_k^1(\mathbb{R})}^{\frac{qs}{2k+2+qs}}\, \||x|^s\,D_k^M(f)\|_{L_k^{q}(\mathbb{R})}^{\frac{2k+2}{2k+2+qs}}, \quad s>0,
 \end{equation*}
 where $q$ is the conjugate exponent of $p$ and $$C(b,q,k) = \left \{ 1+\frac{C_{k,b}^q}{2^{k+1}\Gamma(k+2)}\right\}^{\frac{1}{q }}.$$
 \end{theorem}
 \begin{proof}
    For  $r>0$, we consider
 \begin{equation} \label{e:4.12}
 \|D_k^M(f)\|^q_{L_k^{q}(\mathbb{R})} =    \|\left(1-\chi_{(-r,r)}\right) \, D_k^M(f)\|^{q}_{L_k^{q}(\mathbb{R})}+ \|\chi_{(-r,r)}\, D_k^M(f)\|^{q}_{L_k^{q}(\mathbb{R})}.   
 \end{equation}
By using the fact \eqref{5.1}, we have
 \begin{equation}\label{e:6.2}
 \|\left(1-\chi_{(-r,r)}\right) \, D_k^M(f)\|^{q}_{L_k^{q}(\mathbb{R})}\le \frac{\||x|^s\, D_k^M(f)\|^{q}_{L_k^{q}(\mathbb{R})}}{r^{qs}}.   
 \end{equation}
 On manipulating the second part of \eqref{e:4.12} and invoking
 Riemann-Lebesgue Lemma \ref{le:3.1}, we get
\begin{eqnarray*}
    \|\chi_{(-r,r)}\, D_k^M(f)\|^{q}_{L_k^{q}(\mathbb{R})} &\le& \|\chi_{(-r,r)}\|^{q}_{L_k^{q}(\mathbb{R})}\, \|D_k^M(f)\|^q_{L^\infty_k(\mathbb{R})} \\
    &\le&C_{k,b}^{q}\,\gamma_k\left((-r,r)\right)\,  \|f\|^{q}_{L_k^1(\mathbb{R})},
\end{eqnarray*}
where
\begin{eqnarray} \label{5.9}
\nonumber\gamma_k\left((-r,r)\right) &= &\int_{\mathbb{R}}  \chi_{(-r,r)}(x)\, d\mu_k(x)\\ 
&=& \frac{r^{2k+2}}{2^{k+1}\Gamma(k+2)}.
\end{eqnarray}
Thus, 
\begin{equation}\label{e:6.3}
    \|\chi_{(-r,r)}\,D_k^M(f)\|_{L^{q}_k(\mathbb{R})}^{q} \le C_{k,b}^{q}\,  \frac{r^{2k+2}}{2^{k+1}\Gamma(k+2)}\,\|f\|^{q}_{L_k^1(\mathbb{R})}.
\end{equation}
By combining  \eqref{e:4.12}, \eqref{e:6.2} and \eqref{e:6.3}, we obtain
\begin{equation*}
 \|D_k^M(f)\|_{L_k^{q}(\mathbb{R})}^{q} \le \frac{\||x|^s\, D_k^M(f)\|^{q}_{L_k^{q}(\mathbb{R})}}{r^{qs}}  \,  +  C_{k,b}^{q}\,  \frac{r^{2k+2}}{2^{k+1}\Gamma(k+2)}\,\|f\|^{q}_{L_k^1(\mathbb{R})}. 
\end{equation*}
By choosing $$r= \left[\frac{\||x|^s\, D_k^M(f)\|^{q}_{L_k^{q}(\mathbb{R})} }{\|f\|^{q}_{L_k^1(\mathbb{R})}}\right]^{\frac{1}{qs+2k+2}},$$ we deduce the required result
\begin{equation*}
    \|D_k^M(f)\|_{L_k^{q}(\mathbb{R})}\le C(b,q,k)\, \|f\|_{L_k^1(\mathbb{R})}^{\frac{qs}{2k+2+qs}}\, \||x|^s\, D_k^M(f)\|_{L_k^{q}(\mathbb{R})}^{\frac{2k+2}{2k+2+qs}}.
\end{equation*}
This completes the proof of the theorem.
\begin{remark}
    For $q=2$, the inequality in Theorem \ref{Th:3.4} reduces to \begin{equation*}
    \|f\|_{L_k^{2}(\mathbb{R})}\le C(b,k)\, \|f\|_{L_k^1(\mathbb{R})}^{\frac{s}{k+s+1}}\, \||x|^s\, D_k^M(f)\|_{L_k^{2}(\mathbb{R})}^{\frac{k+1}{k+s+1}}.
\end{equation*} 
\end{remark}

\end{proof}

\begin{theorem} If the function $f$ belongs to $L_k^2(\mathbb{R})$ and $ L_k^p(\mathbb{R}), $  with $1<p<2$, then 
\begin{equation*}
    \|f\|_{L_k^2(\mathbb{R})}\le C(b,q,k)\, \|f\|^{\frac{2sq}{(2k+2)\,(q-2)+2sq}}_{L_k^p(\mathbb{R})}\,\||x|^s\,D_k^M(f)\|_{L_k^2(\mathbb{R})}^{\frac{(2k+2)\,(q-2)}{(2k+2)\,(q-2)+2sq}},\quad s>0,
\end{equation*}
where
\begin{equation*}
 C(b,q,k)=   \left\{1+\left(\frac{C_{k,b}^2}{2^{k+1}\, \Gamma(k+2)}\right)^{\frac{q-2}{q}}\right\}^{\frac{1}{2}}.
\end{equation*}
\end{theorem}
\begin{proof}
    For $r>0$, we consider 
\begin{equation}\label{e:4.19}
\|D_k^M(f)\|^2_{L_k^2(\mathbb{R})} =  \|(1-\chi_{(-r,r)})\, D_k^M(f)\|^2_{L_k^2(\mathbb{R})}+ \|\chi_{(-r,r)}\, D_k^M(f)\|^2_{L_k^2(\mathbb{R})}.
\end{equation}
From \eqref{5.1}, we immediately obtain that
\begin{equation}\label{e:6.4}
    \|(1-\chi_{(-r,r)})\, D_k^M(f)\|^2_{L_k^2(\mathbb{R})} \le \,\frac{\||x|^s\,D_k^M(f)\|^2_{L_k^2(\mathbb{R})}}{r^{2s}}.
\end{equation}
On the other hand, by using H{\"o}lder's inequality, Young's inequality \eqref{eq:3.1} and \eqref{5.9}, we obtain 
\begin{eqnarray} \label{5.14}
\nonumber   \|\chi_{(-r,r)}\, D_k^M(f)\|^2_{L_k^2(\mathbb{R})}&\le& \left(\gamma_k((-r,r))\right)^{\frac{q-2}{q}} \, \|D_k^M(f)\|^2_{L_k^{q}(\mathbb{R})}\\
&\le& \left(\frac{r^{2k+2}\,C_{k,b}^2}{2^{k+1}\, \Gamma(k+2)}\right)^{\frac{q-2}{q}}\,\|f\|^2_{L^p_k(\mathbb{R})}.
\end{eqnarray}
By plugging  \eqref{e:6.4} and \eqref{5.14} in \eqref{e:4.19}, we have \begin{multline*} 
\|D_k^M(f)\|^2_{L_k^2(\mathbb{R})} \le  \,\frac{\||x|^s\,D_k^M(f)\|^2_{L_k^2(\mathbb{R})}}{r^{2s}}\,+ \left(\frac{r^{2k+2}\,C_{k,b}^2}{2^{k+1}\, \Gamma(k+2)}\right)^{\frac{q-2}{q}}\,\|f\|^2_{L^p_k(\mathbb{R})}. 
\end{multline*}
 By choosing \begin{equation*}
    r = \left[\frac{\||x|^s\,D_k^M(f)\|^2_{L_k^2(\mathbb{R})} }{\|f\|^2_{L^p_k(\mathbb{R})}}\right]^{\frac{q}{(2k+2)(q-2)+2sq}},
\end{equation*}
 we deduce that
\begin{equation*}
  \|D_k^M(f)\|_{L_k^2(\mathbb{R})}\le C(b,q,k)\, \|f\|^{\frac{2sq}{(2k+2)\,(q-2)+2sq}}_{L_k^p(\mathbb{R})}\,\||x|^s\,D_k^M(f)\|_{L_k^2(\mathbb{R})}^{\frac{(2k+2)\,(q-2)}{(2k+2)\,(q-2)+2sq}}.   
\end{equation*}
 The required result immediately follows from Plancherel's formula \eqref{e:2.3}.
\end{proof}

\begin{theorem} \label{t:4.8}
Let $1<p_1<p_2\le2$, with $q_1$ and $q_2$ being the conjugate exponents of $p_1$ and $p_2$, respectively. Suppose $f$ belongs to $ L_k^{p_1}(\mathbb{R})$ and $ L_k^{p_2}(\mathbb{R})$. Then there exists a constant $C(q_1,q_2,k,b,s)$ such that
\begin{equation*}
    \|D_k^M(f)\|_{L_k^{q_2}(\mathbb{R})} \le C(q_1,q_2,k,b) \|f\|^{\frac{sq_1q_2}{(2k+2)\,(q_1-q_2)+sq_1q_2}}_{L_k^{p_1}(\mathbb{R})}\||x|^s\, D_k^M(f)\|_{L_k^{q_2}(\mathbb{R})}^{\frac{(2k+2)\,(q_1-q_2)}{(2k+2)\,(q_1-q_2)+sq_1q_2}},
\end{equation*}
for some $s>0$, and
\begin{equation*}
 C(q_1,q_2,k,b) = \left\{ 1+\left(\frac{1}{2^{k+1}\,\Gamma(k+2) }\right)^{\frac{q_1-q_2}{q_1}}\, C_{k,b}^{(1-\frac{2}{q_1})q_2}\right\} ^{\frac{1}{q_{2}}}.  
\end{equation*}

\end{theorem}
\begin{proof}
   For $r>0$, we consider
\begin{equation}\label{e:4.25}
\|D_k^M(f)\|^{q_2}_{L_k^{q_2}(\mathbb{R})}=  \|(1-\chi_{(-r,r)})\,D_k^M(f)\|^{q_2}_{L_k^{q_2}(\mathbb{R})} +   \|\chi_{(-r,r)}\, D_k^M(f)\|^{q_2}_{L_k^{q_2}(\mathbb{R})}.
\end{equation}
In the view of \eqref{5.1}, we have
    \begin{equation} \label{e:6.7}
\|(1-\chi_{(-r,r)})\,D_k^M(f)\|^{q_2}_{L_k^{q_2}(\mathbb{R})} \le \frac{\||x|^s\, D_k^M(f)\| ^{q_2}_{L_k^{q_2}(\mathbb{R})}}{r^{q_2s}}  \, .     
    \end{equation}
Using H{\"o}lder's inequality, Lemma \ref{le:3.2} and  the equation \eqref{5.9}, we obtain
\begin{eqnarray} \label{5.19}
\nonumber  \|\chi_{(-r,r)}\, D_k^M(f)\|^{q_2}_{L_k^{q_2}(\mathbb{R})} &\le& \left(\gamma_k((-r,r))\right)^{\frac{q_1-q_2}{q_1}} \, \|D_k^M(f)\|^{q_2}_{L_k^{q_1}(\mathbb{R})}\\
&\le& \left(\frac{r^{2k+2}}{2^{k+1}\,\Gamma(k+2) }\right)^{\frac{q_1-q_2}{q_1}}\, C_{k,b}^{(1-\frac{2}{q_1})q_2}\,\|f\|^{q_2}_{L_k^{p_1}(\mathbb{R})}.   
\end{eqnarray}
By combining \eqref{e:4.25}, \eqref{e:6.7} and \eqref{5.19}, we have
\begin{multline}\label{5.20}
\|D_k^M(f)\|^{q_2}_{L_k^{q_2}(\mathbb{R})} \le \frac{\||x|^s D_k^M(f)\| ^{q_2}_{L_k^{q_2}(\mathbb{R})}}{r^{q_2s}}  + \left(\frac{r^{2k+2}}{2^{k+1}\Gamma(k+2) }\right)^{\frac{q_1-q_2}{q_1}} C_{k,b}^{(1-\frac{2}{q_1})q_2}\|f\|^{q_2}_{L_k^{p_1}(\mathbb{R})}.
\end{multline}
By setting 
\begin{equation*}
    r= \left[\frac{\||x|^s\, D_k^M(f)\| ^{q_2}_{L_k^{q_2}(\mathbb{R})}}{\|f\|^{q_2}_{L_k^{p_1}(\mathbb{R})}}\right]^{\frac{q_1}{(2k+2)(q_1-q_2)+sq_1q_2}},
\end{equation*}
and substituting into  \eqref{5.20}, we obtain the desired result.
\end{proof}
Now, we will discuss the Clarkson-type  inequalities:
\begin{theorem}
If $f$ belongs to $L^1_k(\mathbb{R})$ and $ L^p_k(\mathbb{R}), 1<p\le 2$, then 
\begin{equation*}
\|f\|_{L^1_k(\mathbb{R})}  \le C(q,k)\,\|f\|_{L_k^p(\mathbb{R})}^{\frac{qs}{2k+2+qs}}\, \||y|^s\,f\|_{L^1_k(\mathbb{R})}^{\frac{2k+2}{2k+2+qs}}, \,\,\, s>0 
\end{equation*}
where
\begin{equation*}
    C(q,k) = \left\{ \frac{1}{(2^{k+1}
    \,\Gamma(k+2)\,)^{\frac{1}{q}}}+ 1\right\}.
\end{equation*} 
\end{theorem}
\begin{proof}
     For  $r>0$, we consider the following inequality
\begin{equation}\label{e:4.16}
 \|f\|_{L^1_k(\mathbb{R})} \le   \|\left(1-\chi_{(-r,r)}\right)f\|_{L_k^1(\mathbb{R})}+  \|\chi_{(-r,r)}\,f\|_{L^1_k(\mathbb{R})}.
 \end{equation}
 Using the fact \eqref{5.1}, we obtain
\begin{equation} \label{5.11}
 \|\left(1-\chi_{(-r,r)}\right)f\|_{L_k^1(\mathbb{R})}\le \frac{\||y|^s\,f\|_{L^1_k(\mathbb{R})}}{r^s}.  
\end{equation}
By applying H{\"o}lder's inequality and  \eqref{5.9}, we have
\begin{eqnarray} \label{5.12}
 \nonumber\|\chi_{(-r,r)}\,f\|_{L^1_k(\mathbb{R})}&\le& \|\chi_{(-r,r)}\|_{L_k^q(\mathbb{R})}\, \|f\|_{L_k^p(\mathbb{R})}\\
&=& \left[\frac{r^{2k+2}}{2^{k+1}\Gamma(k+2)}\right]^{\frac{1}{q}}\, \|f\|_{L_k^p(\mathbb{R})}
\end{eqnarray}
From inequalities \eqref{e:4.16}, \eqref{5.11} and \eqref{5.12}, we obtain that
\begin{equation*}
 \|f\|_{L^1_k(\mathbb{R})} \le  \frac{\||y|^s\,f\|_{L^1_k(\mathbb{R})}}{r^s}+\left[\frac{r^{2k+2}}{2^{k+1}\Gamma(k+2)}\right]^{\frac{1}{q}}\, \|f\|_{L_k^p(\mathbb{R})}.   
\end{equation*}
By choosing \begin{equation*}
 r= \left[\frac{\||y|^s\,f\|_{L^1_k(\mathbb{R})} }{\|f\|_{L_k^p(\mathbb{R})}} \right]^{\frac{q}{2k+2+qs}},  
\end{equation*}
we derive that
\begin{equation*}
    \|f\|_{L^1_k(\mathbb{R})}  \le C(q,k)\,\|f\|_{L_k^p(\mathbb{R})}^{\frac{qs}{2k+2+qs}}\, \||y|^sf\|_{L^1_k(\mathbb{R})}^{\frac{2k+2}{2k+2+qs}}. 
\end{equation*}
This concludes the proof.

\end{proof}
\begin{theorem} If   $f$ belongs to  $L^2_k(\mathbb{R})$ and $ L_k^p(\mathbb{R}) $, $1<p<2, $ then 
\begin{equation*}
 \|f\|_{L_k^p(\mathbb{R})}    \le C(k,p) \|f\|^{\frac{ps}{(k+1)(2-p)+ps}}_{L_k^2(\mathbb{R})} \||y|^s\, f\|^{\frac{(k+1)(2-p)}{(k+1)(2-p)+ps}}_{L_k^p(\mathbb{R})},~ s>0
\end{equation*}
where
\begin{equation*}
C(k,p) = \left\{ 1+\frac{1}{(2^{k+1}\,\Gamma(k+2))^{\frac{2-p}{2}}}\right\}^{\frac{1}{p}}.
\end{equation*}
\end{theorem}
\begin{proof}
 For $r>0$, we have
\begin{equation}\label{e:6.5}
\|f\|^p_{L_k^p(\mathbb{R})} = \|\chi_{(-r,r)}\, f\|^p_{L_k^p(\mathbb{R})}+\|(1-\chi_{(-r,r)})f\|^p_{L_k^p(\mathbb{R})}.
\end{equation}
From  \eqref{5.1}, we obtain that
\begin{equation}\label{e:6.6}
\|(1-\chi_{(-r,r)})f\|^p_{L_k^p(\mathbb{R})} \le \frac{\||y|^s\,f\|^p_{L_k^p(\mathbb{R})}}{r^{ps}}.   
\end{equation}
By using H{\"o}lder's inequality and \eqref{5.9}, we obtain    
\begin{eqnarray} \label{5.17}
\nonumber \|\chi_{(-r,r)}\, f\|^p_{L_k^p(\mathbb{R})} &\le& \left(\gamma_k((-r,r))\right)^{\frac{2-p}{2}} \, \|f\|^p_{L_k^2(\mathbb{R})}   \\
\nonumber &\le& \left(\frac{r^{2k+2}}{2^k\,(2k+2)\Gamma(k+1)}\right)^{\frac{2-p}{2}}\,\|f\|^p_{L_k^2(\mathbb{R})}\\
 &=& \frac{r^{(k+1)\,(2-p)}}{(2^{k+1}\,\Gamma(k+2))^{\frac{2-p}{2}}}\,\|f\|^p_{L_k^2(\mathbb{R})}.
\end{eqnarray}
On substituting   \eqref{e:6.6} and \eqref{5.17} in \eqref{e:6.5}, we get 
\begin{equation*}
\|f\|^p_{L_k^p(\mathbb{R})} \le \, \frac{\||y|^s\,f\|^p_{L_k^p(\mathbb{R})}}{r^{ps}}\,+\,\frac{r^{(k+1)\,(2-p)}}{(2^{k+1}\,\Gamma(k+2))^{\frac{2-p}{2}}}\,\|f\|^p_{L_k^2(\mathbb{R})}.
\end{equation*}
By choosing
\begin{equation*}
    r = \left[\frac{ \||y|^s\,f\|^p_{L_k^p(\mathbb{R})}}{ \|f\|^p_{L_k^2(\mathbb{R})}}\right]^{\frac{1}{(k+1)\,(2-p)+ps}},
\end{equation*}
 we deduce that 
\begin{equation*}
 \|f\|_{L_k^p(\mathbb{R})}    \le C(k,p)\, \||y|^s\, f\|^{\frac{(k+1)(2-p)}{(k+1)(2-p)+ps}}_{L_k^p(\mathbb{R})}\, \|f\|^{\frac{ps}{(k+1)(2-p)+ps}}_{L_k^2(\mathbb{R})} .   
\end{equation*}
\end{proof}
\begin{theorem}\label{t:4.9}
If $f$ belongs to $ L_k^{p_1}(\mathbb{R})$ and $ L_k^{p_2}(\mathbb{R}) $, where $1<p_1<p_2\le 2$,  then there exists constant $C(p_1,p_2,k)$ such that 
\begin{equation*}
    \|f\|_{L_k^{p_1}(\mathbb{R})} \le C(p_1,p_2,k) \|f\|_{L_k^{p_2}(\mathbb{R})}^{\frac{p_1p_2s}{(2k+2)(p_2-p_1)+p_1p_2s}}\||y|^s\,f\|^{\frac{(2k+2)\,(p_2-p_1)}{(2k+2)(p_2-p_1)+p_1p_2s}}_{L_k^{p_1}(\mathbb{R})},
\end{equation*}
where,
\begin{equation*}
  C(p_1,p_2,k) = \left\{ 1+\frac{1}{(2^{k+1}\,\Gamma(k+2))^{\frac{p_2-p_1}{p_2}}}\right\}^{\frac{1}{p_1}}.  
\end{equation*}
\begin{proof}
Let $r>0$ and $f\in L_k^{p_1}(\mathbb{R})\cap L_k^{p_2}(\mathbb{R}),$ with $ 1<p_1<p_2\le 2$. Then
\begin{equation}\label{e:4.29}
 \|f\|^{p_1}_{L_k^{p_1}(\mathbb{R})}=\|\left(1-\chi_{(-r,r)}\right) f\|^{p_1}_{L_k^{p_1}(\mathbb{R})}+  \|\chi_{(-r,r)}f\|^{p_1}_{L_k^{p_1}(\mathbb{R})}. 
\end{equation}
From  equation \eqref{5.1}, we have
\begin{equation} \label{e:6.8}
\|\left(1-\chi_{(-r,r)}\right) f\|^{p_1}_{L_k^{p_1}(\mathbb{R})}\le \frac{\||y|^s\,f\|^{p_1}_{L_k^{p_1}}(\mathbb{R})}{r^{p_1s}} .  
\end{equation}
Applying H{\"o}lder's inequality and using equation \eqref{5.9},  we obtain 
\begin{eqnarray} \label{5.21}
\nonumber \|\chi_{(-r,r)}f\|^{p_1}_{L_k^{p_1}(\mathbb{R})}&\le& \left(\gamma_k(-r,r)\right)^{\frac{p_2-p_1}{p_2}}\, \|f\|^{p_1}_{L_k^{p_2}(\mathbb{R})} \\
 &\le& \left(\frac{r^{(2k+2)}}{2^{k+1}\,\Gamma(k+2)\, }\right)^{\frac{p_2-p_1}{p_2}}\, \|f\|^{p_1}_{L_k^{p_2}(\mathbb{R})}.
\end{eqnarray}
By combining  \eqref{e:4.29}, \eqref{e:6.8} and \eqref{5.21}, we have
\begin{equation*}
    \|f\|^{p_1}_{L_k^{p_1}(\mathbb{R})} \le \frac{\||y|^s\,f\|^{p_1}_{L_k^{p_1}(\mathbb{R})}}{r^{p_1s}} +\left(\frac{r^{(2k+2)}}{2^{k+1}\,\Gamma(k+2)\, }\right)^{\frac{p_2-p_1}{p_2}}\, \|f\|^{p_1}_{L_k^{p_2}(\mathbb{R})}.
\end{equation*}
By choosing
\begin{equation*}
    r= \left[\frac{\||y|^s\,f\|^{p_1}_{L_k^{p_1}(\mathbb{R})}}{\|f\|^{p_1}_{L_k^{p_2}(\mathbb{R})}}\right]^{\frac{p_2}{(2k+2)(p_2-p_1)+p_1p_2s}},
\end{equation*}
we deduce that
 \begin{equation*}
    \|f\|_{L_k^{p_1}(\mathbb{R})} \le C(p_1,p_2,k)\, \||y|^s\,f\|^{\frac{(2k+2)\,(p_2-p_1)}{(2k+2)(p_2-p_1)+p_1p_2s}}_{L_k^{p_1}(\mathbb{R})}\, \|f\|_{L_k^{p_2}(\mathbb{R})}^{\frac{p_1p_2s}{(2k+2)(p_2-p_1)+p_1p_2s}}.
\end{equation*}   
\end{proof}    
\end{theorem}
\end{subsection}
\end{subsection}
\begin{subsection}{Donoho-Stark and Matolcsi-Szucs Uncertainty}In this subsection, we discuss the
extended version of the classical Heisenberg uncertainty principle known as the Donoho-Stark \cite{donoho1989uncertainty} and Matolcsi-Szucs uncertainty principle. It states that  $f$ and  $\hat{f}$ are essentially zero outside a measurable set $E$ and  $F$ then $|E|\,|F| \ge 1- \delta$, where $|E|$ and $|F|$ denotes the measures of sets $E$ and $F$, and $\delta>0$ is a small number.  It has applications in image processing and signal reconstruction, compression, denoising, etc. Here, $q,q_1$, and $q_2$  are conjugate components of $p,p_1$, and $p_2$ respectively.
\begin{definition}
Let $E$  be a measurable subset of $\mathbb{R}$. The function $f\in L_k^p(\mathbb{R})$, $1\le p\le 2$  is said to be $\epsilon_E$ - concentrated to $E$ in $L_k^p(\mathbb{R})$, if
\begin{equation*}\label{6.1}
    \|f-\chi_Ef\|_{L^p_k(\mathbb{R})}\le \epsilon_E\, \|f\|_{L^p_k(\mathbb{R})},\qquad\qquad \text{for}\,   \,\,0\le \epsilon_E <1.
\end{equation*}
Similarly, we say that $D_k^M(f)$ is $\epsilon_E$- concentrated to $E$ in $L_k^{q}(\mathbb{R}),$ if 
\begin{equation*}\label{6.2}
    \|D_k^M(f)- \chi_ED_k^M(f)\|_{L_k^{q}(\mathbb{R})} \le \epsilon_E\, \|D_k^M(f)\|_{L_k^{q}(\mathbb{R})}.
\end{equation*}
\end{definition}
We discuss the Donoho-Stark inequalities in the following theorems.
 \begin{theorem}
  Let $E$ and $F$ be measurable subsets of $\mathbb{R}$, $f$ belongs to $L_k^1(\mathbb{R})$ and $L^p_k(\mathbb{R})$ for  $1<p\le 2 $. If $f$ is $\epsilon_E$ concentrated to $E$ in $L_k^1(\mathbb{R})$ and $D_k^M(f)$ is $\epsilon_{F}$ concentrated to $F$ in $L_k^{q}(\mathbb{R})$, then 
 \begin{equation*}
     \|D_k^M(f)\|_{L_k^{q}(\mathbb{R})}\le C_{k,b}\,\frac{(\gamma_k(F))^{\frac{1}{q}}\, (\gamma_k(E))^{\frac{1}{q}}}{(1-\epsilon_E)\,(1-\epsilon_{F})}\,\|f\|_{L_k^{p}(\mathbb{R})}.
 \end{equation*}
\end{theorem} 
\begin{proof}
    Let $f$ belongs to $L_k^1(\mathbb{R})$ and $ L_k^p(\mathbb{R})$ for $1<p\le 2$.  Suppose that $\gamma_k(E)$ and $\gamma_k(F)$ are finite. By using the fact that $D_k^M(f)$ is $\epsilon_{F}$ concentrated to $F$ in $L_k^{q}(\mathbb{R})$ and invoking the Reimann-Lebesgue Lemma \ref{le:3.1}, we get the following result
\begin{eqnarray}\label{6.3}
\nonumber        \|D_k^M(f)\|_{L_k^{q}(\mathbb{R})}&\le&\|D_k^M(f)-\chi_{F}\, D_k^M(f)\|_{L_k^{q}(\mathbb{R})}+\|\chi_{F}\, D_k^M(f)\|_{L^{q}_k(\mathbb{R})}
\\
\nonumber &\le& \epsilon_{F}\,\|D_k^M(f)\|_{L_k^{q}(\mathbb{R})}+(\gamma_k({F}))^{\frac{1}{q}}\, \|D_k^M(f)\|_{L_k^\infty(\mathbb{R})}\\
\nonumber        &\le& \epsilon_{F}\,\|D_k^M(f)\|_{L_k^{q}(\mathbb{R})}+(\gamma_k({F}))^{\frac{1}{q}}\,C_{k,b}\, \|f\|_{L_k^1(\mathbb{R})} \\
&\le&C_{k,b}\, \frac{(\gamma_k(F))^{\frac{1}{q}}}{1-\epsilon_{F}}\, \|f\|_{L^1_k(\mathbb{R})}.
\end{eqnarray}
On the other hand, $f$ is $\epsilon_E$ concentrated to $E$ in $L_k^1(\mathbb{R})$. Thus
\begin{eqnarray} \label{6.4}
\nonumber    \|f\|_{L_k^1(\mathbb{R})} &\le& \|f-\chi_Ef\|_{L_k^1(\mathbb{R})}+\|\chi_Ef\|_{L_k^1(\mathbb{R})}\\
    &\le& \frac{(\gamma_k(E))^{\frac{1}{q}}}{1-\epsilon_E}\,\|f\|_{L_k^p(\mathbb{R})}.
\end{eqnarray}
Substituting \eqref{6.4} in \eqref{6.3}, we get
\begin{equation*}
     \|D_k^M(f)\|_{L_k^{q}(\mathbb{R})}\le C_{k,b}\,\frac{(\gamma_k(F))^{\frac{1}{q}}\, (\gamma_k(E))^{\frac{1}{q}}}{(1-\epsilon_E)\,(1-\epsilon_{F})}\,\|f\|_{L_k^{p}(\mathbb{R})}.
 \end{equation*}
 This completes the proof of the Theorem.
 \end{proof}
 \begin{theorem}
  Let $E$ and $F$ are measurable subsets of $\mathbb{R}$, and let $ f\in L_k^{p_1}(\mathbb{R})\cap L^{p_2}_k(\mathbb{R})$ for $1<p_1<p_2\le 2$. If $f$ is $\epsilon_E$-concentrated in $L_k^{p_1}(\mathbb{R})$ and $D_k^M(f)$ is $\epsilon_{F}$ concentrated in $L_k^{q_2}(\mathbb{R})$, then 
  \begin{equation*}
      \|D_k^M(f)\|_{L_k^{q_2}(\mathbb{R})} \le C_{k,b}^{1-\frac{2}{q_1}}\, \frac{(\gamma_k(E))^{\frac{p_2-p_1}{p_1\,p_2}}\, (\gamma_k(F))^{\frac{q_1-q_2}{q_1\,q_2}}}{(1-\epsilon_E)\,(1-\epsilon_{F})}\, \|f\|_{L_k^{p_2}(\mathbb{R})}.
  \end{equation*}
 \end{theorem}
\begin{proof}
    Suppose that $\gamma_k(E)$ and $\gamma_k(F)$ are finite. Let $f\in L_k^{p_1}(\mathbb{R})\cap L_k^{p_2}(\mathbb{R})$ for $1<p_1\le p_2 \le 2.$ Since $D_k^M(f)$ is $\epsilon_{F}$  concentrated to $F$ in $L_k^{q_2}(\mathbb{R})$. Then by H{\"o}lder's inequality, we have
    \begin{eqnarray}\label{e:7.1}
\nonumber \|D_k^M(f)\|_{L_k^{q_2}(\mathbb{R})} &\le& \epsilon_{F}\,  \|D_k^M(f)\|_{L_k^{q_2}(\mathbb{R})} +\| \chi_{F}\,  D_k^M(f)\|_{L_k^{q_2}(\mathbb{R})}\\
 &\le&   C_{k,b}^{1-\frac{2}{q_1}}\, \frac{(\gamma_k(F))^{\frac{q_1-q_2}{q_1q_2}}}{1-\epsilon_{F}}\,\|f\|_{L_k^{p_1}(\mathbb{R})}.   
 \end{eqnarray} 
On the other hand, $f$ is $\epsilon_E$ concentrated to $E$ in $L_k^{p_1}(\mathbb{R})$, and applying H{\"o}lder's inequality, we obtain  
\begin{eqnarray}
\nonumber    \|f\|_{L_k^{p_1}(\mathbb{R})} &\le &\epsilon_E\,\|f\|_{L_k^{p_1}(\mathbb{R})}+\|\chi_E\,f\|_{L_k^{p_1}(\mathbb{R})}\\
 &\le& \frac{(\gamma_k(E))^{\frac{p_2-p_1}{p_2\,p_1}}}{1-\epsilon_E}  \, \|f\|_{L_k^{p_2}(\mathbb{R})}.\label{e:7.2}
\end{eqnarray} 
By combining equations \eqref{e:7.1} and \eqref{e:7.2}, we obtain the desired result. 
\end{proof}
\begin{definition} \label{d:4.14}
    Let $E$ be a measurable subset of $\mathbb{R}$. We say that the function $f\in L_k^p(\mathbb{R})$, $1\le p\le 2 $ is $\epsilon_E$ bandlimited to $E$ in  $L_k^p(\mathbb{R}),$ if there is a function $h\in B^p(E)$ such that 
    \begin{equation*}\label{E:6.3}
        \|f-h\|_{L_k^p(\mathbb{R})} \le \epsilon_E\, \|f\|_{L_k^p(\mathbb{R})}, \qquad 0\le \epsilon_E <1,
    \end{equation*}
    where,
    \begin{equation*}
 B^p(E) = \left\{ h\in L_k^p(\mathbb{R}) :\,\chi_E\, D_k^M(h)= D_k^M(h)\right\} .      
    \end{equation*}
\end{definition}

\begin{theorem}
    Let $E$ and $F$ are measurable subsets of $\mathbb{R}$ and  $f \in L_k^{p_1}(\mathbb{R}) \cap L_k^{p_2}(\mathbb{R})$, $1< p_1 \le p_2 \le 2$. If $f$ is $\epsilon_{E}$ concentrated to $E$ in $L_k^{p_1}(\mathbb{R})$ and $\epsilon_{F}$ bandlimited to $F$ in $L_k^{p_2}(\mathbb{R})$, then
    \begin{equation*}
\|f\|_{L_k^{p_1}(\mathbb{R})} \le \frac{(\gamma_k(E))^{\frac{p_2-p_1}{p_1\,p_2}}}{1-\epsilon_E}  \,  \left[\frac{(1+\epsilon_{F})\,(\gamma_k(E))^{\frac{1}{p_2}}\, (\gamma_k(F))^{\frac{1}{p_2}}}{|b|^{2(k+1)\left(1-\frac{1}{q_2}\right)}}+\epsilon_{F}\right] \, \|f\|_{L_k^{p_2}(\mathbb{R})} .
    \end{equation*}
\end{theorem}
\begin{proof}
    Suppose that $\gamma_k(E)$ and $\gamma_k(F)$ are finite. Let $f\in L_k^{p_1}(\mathbb{R}) \cap L_k^{p_2}(\mathbb{R})$, $1<p_1<p_2\le 2.$ Since $f$ is $\epsilon_E$ concentrated to $E$ in $L_k^p(\mathbb{R})$, then by H{\"o}lder's inequality, we have
\begin{eqnarray}
\nonumber\|f\|_{L_k^{p_1}(\mathbb{R})} &\le& \epsilon_E\,\|f\|_{L_k^{p_1}(\mathbb{R})}+\|\chi_E\,f\|_{L_k^{p_1}(\mathbb{R})}\\
 \label{e:7.3}
  &\le&  \frac{1}{1-\epsilon_E} \,(\gamma_k(E))^{\frac{p_2-p_1}{p_1\,p_2}}  \, \|\chi_E\,f\|_{L_k^{p_2}(\mathbb{R})}.
\end{eqnarray}
    By assumption, $f$ is $\epsilon_{F}$ bandlimited to $F$ in $L_k^{p_2}(\mathbb{R})$, and from  the Definition \ref{d:4.14}, we have
 \begin{eqnarray}
\nonumber \|f-h\|_{L_k^{p_2}(\mathbb{R})}  &\le &\epsilon_{F} \, \|f\|_{L_k^{p_2}(\mathbb{R})}, \quad \text{for}\,\,\,\, h\in B^{p_2}(F).  
 \end{eqnarray}
Consequently,
 \begin{eqnarray*}
\|h\|_{L_k^{p_2}(\mathbb{R})}  &\le &(1+\epsilon_{F})\, \|f\|_{L_k^{p_2}(\mathbb{R})}.
 \end{eqnarray*}
For this $h$, we have
\begin{eqnarray}
\nonumber    \|\chi_E\,f\|_{L_k^{p_2}(\mathbb{R})}&\le& \|\chi_E\,h\|_{L_k^{p_2}(\mathbb{R})}+\|\chi_E(f-h)\|_{L_k^{p_2}(\mathbb{R})}\\
    &\le& \|\chi_E\,h\|_{L_k^{p_2}(\mathbb{R})}+\epsilon_F\|f\|_{L_k^{p_2}(\mathbb{R})}.\label{e:7.4}
\end{eqnarray}
Since $h\in B^{p_2}(F)$, and using 
 \eqref{e:2.2}, we have
\begin{eqnarray*}
  D_k^{M}(h)(t) &=& \chi_{F}\, D_k^M(h)(t)\\
h(t) &=& D_k^{M^{-1}}(\chi_{F}\, D_k^M(h))(t).
\end{eqnarray*}
By applying H{\"o}lder's inequality and Young's inequality \eqref{le:3.2}, we deduce that
\begin{eqnarray*}
    |h(t)|&\le &\frac{(\gamma_k(F))^{\frac{1}{p_2}}}{|b|^{k+1}}\,\|D_k^M(f)\|_{L_k^{q_2}(\mathbb{R})}\\
    &\le& \frac{(\gamma_k(F))^{\frac{1}{p_2}}}{|b|^{2(k+1)\left(1-\frac{1}{q_2}\right)}}\, \|f\|_{L_k^{p_2}(\mathbb{R})}.
\end{eqnarray*}
Thus, 
\begin{equation} \label{e:7.5}
\|\chi_E\,h\|_{L_k^{p_2}(\mathbb{R})} \le \frac{(\gamma_k(F))^{\frac{1}{p_2}}\,(\gamma_k(E))^{\frac{1}{p_2}}}{|b|^{2(k+1)\left(1-\frac{1}{q_2}\right)}}\, \|f\|_{L_k^{p_2}(\mathbb{R})}.   
\end{equation}
Substituting  \eqref{e:7.5} in \eqref{e:7.4} , we get
\begin{equation}\label{e:7.6}
    \|\chi_E\,f\|_{L_k^{p_2}(\mathbb{R})}\le  \left[\frac{(1+\epsilon_{F})\,(\gamma_k(E))^{\frac{1}{p_2}}\, (\gamma_k(F))^{\frac{1}{p_2}}}{|b|^{2(k+1)\left(1-\frac{1}{q_2}\right)}}+\epsilon_{F}\right] \, \|f\|_{L_k^{p_2}(\mathbb{R})} .
\end{equation}
Therefore, by \eqref{e:7.6} and \eqref{e:7.3}, we obtain the desired result.
\end{proof}
Now, we establish the Matolcsi-Szucs-type inequalities.
\begin{theorem} Let $1<p\le 2$ and $f\in   L_k^1(\mathbb{R})\cap L_k^p(\mathbb{R})$.  Then
\begin{equation*}
    \|D_k^M(f)\|_{L_k^{q}(\mathbb{R})} \le C_{k,b}\, (\gamma_k(A_{D_k^M(f)}))^{\frac{1}{q}}\, (\gamma_k(A_f))^{\frac{1}{q}}\, \|f\|_{L_k^1(\mathbb{R})},
\end{equation*}
where  $$A_{D_k^M(f)}= \left\{x\in \mathbb{R}: D_k^M(f)(x) \neq 0\right\}$$ and $$A_f = \left\{ t
\in \mathbb{R}:f(t)\neq0\right\}.$$
\end{theorem}
\begin{proof}
    Let us consider $E=A_{D_k^M(f)}$. By Reimann-Lebesgue Lemma \ref{le:3.1} and H{\"o}lder's inequality, we have
    \begin{eqnarray*}    \|D_k^M(f)\|_{L^{q}_k(\mathbb{R})}&=& \|\chi_E\, D_k^M(f)\|_{L_k^{q}(\mathbb{R})}\\
        &\le& \|\chi_E\|_{L_k^{q}(\mathbb{R})}\,\|D_k^M(f)\|_{L_k^{\infty}(\mathbb{R})}\\
        &\le& \frac{ (\gamma_k(E))^{\frac{1}{q}}}{|b|^{k+1}}\|f\|_{L_k^1(\mathbb{R})}\\
        &=& \frac{ (\gamma_k(E))^{\frac{1}{q}}}{|b|^{k+1}}\|f\,\chi_{A_f}\|_{L_k^1(\mathbb{R})}\\
        &\le&  C_{k,b}\,(\gamma_k(E))^{\frac{1}{q}}\, (\gamma_k(A_f))^{\frac{1}{q}}\, \|f\|_{L_k^p(\mathbb{R})}.
    \end{eqnarray*}
    Thus, we conclude the proof.
\end{proof}
\begin{theorem}
If the function  $f$ belongs to $ L_k^{p_1}(\mathbb{R})$ and $L_k^{p_2}(\mathbb{R})$, where  $1<p_1\le p_2 \le 2$, then 
\begin{equation*}
     \|D_k^M(f)\|_{L_k^{q_2}(\mathbb{R})} \le C_{k,b}^{1-\frac{2}{q_1}}\, (\gamma_k(A_{D_k^M(f)}))^{\frac{q_1-q_2}{q_1\,q_2}}\, (\gamma_k(A_f))^{\frac{p_1-p_2}{p_1\,p_2}}\, \|f\|_{L^{p_2}_k(\mathbb{R})},
\end{equation*}
where $$A_{D_k^M(f)}= \left\{x\in \mathbb{R}: D_k^M(f)(x) \neq 0\right\}$$ and $$A_f = \left\{ t
\in \mathbb{R}:f(t)\neq0\right\}.$$
\end{theorem}
\begin{proof}
    Let  $ 1<p_1\le p_2 \le 2$, and let $f\in  L_k^{p_1}(\mathbb{R})\cap L_k^{p_2}(\mathbb{R})$. Set $E=A_{D_k^M(f)}$. By using H{\"o}lder's inequality and Young's inequality \eqref{le:3.2}, we obtain 
    \begin{eqnarray*}
        \|D_k^M(f)\|_{L_k^{q_2}(\mathbb{R})} &=& \|\chi_E\, D_k^M(f)\|_{L_k^{q_2}(\mathbb{R})}\\
        &\le&  (\gamma_k(E))^{\frac{q_1-q_2}{q_1\,q_2}}\, \|D_k^M(f)\|_{L^{q_1}_k(\mathbb{R})}\\
        &\le& C_{k,b}^{1-\frac{2}{q_1}}\,(\gamma_k(E))^{\frac{q_1-q_2}{q_1\,q_2}}\, \|f\|_{L^{p_1}_k(\mathbb{R})}\\
         &=& C_{k,b}^{1-\frac{2}{q_1}}\,(\gamma_k(E))^{\frac{q_1-q_2}{q_1\,q_2}} \|f\,\chi_{A_f}\|_{L^{p_1}_k(\mathbb{R})}\\
        &\le&   C_{k,b}^{1-\frac{2}{q_1}}\,(\gamma_k(A_{D_k^M(f)}))^{\frac{q_1-q_2}{q_1\,q_2}}\, (\gamma_k(A_f))^{\frac{p_1-p_2}{p_1\,p_2}}\, \|f\|_{L^{p_2}_k(\mathbb{R})}.
    \end{eqnarray*}
    This completes the proof.
\end{proof}
\end{subsection}
\section{Qualitative Uncertainty Principle} \label{S:3}
 
  In this section, we will discuss the Miyachi and Cowling-Price qualitative uncertainty principles for the linear canonical Dunkl transform (LCDT). 
\subsection{Miyachi uncertainty principle}
\begin{theorem} \label{Th: 4.1}
Let $f$ be a measurable function on $\mathbb{R}$ such that $e^{sx^2}f\in L^p_k(\mathbb{R})+L^{q}_k(\mathbb{R}),$ where $  p,q\in [1, \infty]$. Suppose that
\begin{equation}\label{4.4}
\int_{\mathbb{R}} ln^+\left( \frac{|e^{tx^2} D_k^M(f)(x)|}{\lambda}\right) dx< \infty,    
\end{equation}
for some positive real constants $s,t,\lambda$, and 
\begin{equation*}
ln^+(r)=\left \{\begin{array}{ll}
ln(r)& 
\text{if}\,\, r>1\\
0, & \text{otherwise}.
\end{array}\right. 
\end{equation*} 
Then 
\begin{itemize}
    \item [$(i)$]  $f=0$  a.e on $\mathbb{R}$, if $st>\frac{1}{4b^2}$.
    \item[$(ii)$] $f(x) =  C\, e^{-\left(\frac{i}{2}\frac{a}{b}+s\right)x^2}$, if $st= \frac{1}{4b^2}$ and $|C|\le |\lambda|$.
    \item[$(iii)$] There are many functions satisfying the hypothesis \eqref{4.4}, if $st<\frac{1}{4b^2}$.
\end{itemize}
\end{theorem}
In order to simplify the proof of Miyachi's theorem, we first recall the Lemma \ref{Lm: 4.2} and establish the Lemma \ref{Lm: 4.3}:
\begin{lemma} \label{Lm: 4.2} \cite{chouchene2011miyachi}
    Let $h$ be an entire function on $\mathbb{C}$ such that 
    \begin{eqnarray*}
    |h(z)| \le C e^{t\,{(\mathfrak{Re}z)}^2}\quad\text{and} \quad
\int_{\mathbb{R}} ln^+ |h(x)| dx < \infty,\quad 
\end{eqnarray*} 
for some constants $C>0$ and $ ~t>0$.
Then $h$ is a constant.\label{l:4.1}
\end{lemma}
\begin{lemma} \label{Lm: 4.3}
If $f$ be a measurable function on $\mathbb{R}$ such that $$e^{sx^2}f\in L^p_k(\mathbb{R})+L^{q}_k(\mathbb{R}), \quad 1\le p,q \le \infty,$$
and $s>0$,  then $e^{-\frac{i}{2}\frac{d}{b}(\cdot)^2}D_k^M(f)$ is an entire function on $\mathbb{C}$ satisfying
\begin{equation*}
|e^{-\frac{i}{2}\frac{d}{b}z^2}D^M_k(f)(z)|\le Ce^{\frac{(\mathfrak{Im}z)^2}{4b^2s}}, \quad z\in \mathbb{C}. 
\end{equation*}\label{l:4.2}
\end{lemma}
\begin{proof}
Let $e^{sx^2} f\in L_k^p(\mathbb{R})+L_k^{q}(\mathbb{R})$. Then there exist two functions $f_1\in L_k^p(\mathbb{R})$ and $f_2\in L_k^{q}(\mathbb{R})$ such that 
\begin{equation}\label{4.1}
f(x)=e^{-sx^2}(f_1(x)+f_2(x)), \quad \forall x \in \mathbb{R}.   
\end{equation}
On integrating \eqref{4.1} with respect to  $ d\mu_k(x)$ and applying H\"older's inequality we can observe that  $f$ is an integrable function on $\mathbb{R}$. Thus, $D_k^M(f)$ is well defined. By recalling Note \ref{N:2.4}, we see that    $e^ {-\frac{i}{2}\frac{d}{b}(\cdot)^2}D_k^M(f)$ is an entire function. From Property \ref{pr:2.2} (ii), we have
 \begin{equation} \label{e:4.4}
\left|E_k\left(-\frac{iz}{b},x\right)\right| \le e^{\frac{|\mathfrak{Im}z|}{|b|}|x|}.
\end{equation}
Now, invoking  Definition \ref{D:1}, \eqref{4.1} and the fact \eqref{e:4.4} we obtain
\begin{eqnarray} \label{e:3.5}
 \nonumber    |e^{-\frac{i}{2}\frac{d}{b}z^2} D_k^M(f)(z)| 
  \le&\frac{1}{|(ib)^{k+1}|}  \int_{\mathbb{R}} |e^{-sx^2}| | f_1(x)| e^{\frac{|\mathfrak{Im}z|}{|b|}|x|} d\mu_k(x)\\
  &+\frac{1}{|(ib)^{k+1}|}\int_{\mathbb{R}} |e^{-sx^2}| | f_2(x)| e^{\frac{|\mathfrak{Im}z|}{|b|}|x|} d\mu_k(x).
 \end{eqnarray}
Thus, by using H\"older's  inequality we have
 \begin{eqnarray}
 \int_{\mathbb{R}} |e^{-sx^2}| | f_1(x)| e^{\frac{|\mathfrak{Im}z|}{|b|}|x|} d\mu_k(x) &\le& \frac{C}{|b|^{k+1}} e^{\frac{(\mathfrak{Im}z)^2}{4sb^2}} \|f_1\|_{L_k^p(\mathbb{R})}, \label{e:4.1}\\
 \int_{\mathbb{R}} |e^{-sx^2}| | f_2(x)| e^{\frac{|\mathfrak{Im}z|}{|b|}|x|} d\mu_k(x) &\le& \frac{C}{|b|^{k+1}} e^{\frac{(\mathfrak{Im}z)^2}{4sb^2}} \|f_2\|_{L_k^{q}(\mathbb{R})} \label{e:4.2}.
 \end{eqnarray}
Hence by \eqref{e:3.5}, \eqref{e:4.1} and \eqref{e:4.2}, we get the desired result.
\end{proof}

\begin{proof} {\bf of Theorem \ref{Th: 4.1}.}
Let $1\le p,q\le \infty$, and let $e^{sx^2}f\in L^p_k(\mathbb{R})+L^{q}_k(\mathbb{R})$ for $s>0$. We define 
\begin{equation}\label{e:3.9}
h(z) = e^{\frac{z^2}{4sb^2}}\, e^{-\frac{i}{2}\frac{d}{b}z^2}\,D_k^M(f)(z), \,\, z\in \mathbb{C}.
\end{equation}
Then by using Lemma \ref{l:4.2}, we have
 $$|h(z)| \le Ce^{\frac{(\mathfrak{Re}z)^2}{4b^2s}},\,\, z\in \mathbb{C}.$$
\begin{itemize}
        \item [$(i)$.] Now, using the fact  $ln^+(xy)\le ln^+(x)+y$ for $x,y>0$, we deduce that 
 \begin{eqnarray*}
     \int_{\mathbb{R}}ln^+(|h(x)|)\, dx&=& \int_{\mathbb{R}}ln^+\left(|e^{\frac{1}{4sb^2}x^2}\, e^{-\frac{i}{2}\frac{d}{b}x^2}\,D_k^M(f)(x)|\right)\, dx\\
     &\le& \int_{\mathbb{R}} ln^+\left(\frac{|e^{tx^2}D_k^M(f)(x)|}{\lambda}\right)\, dx+\int_{\mathbb{R}} \lambda \,e^{\left(\frac{1}{4sb^2}-t \right)x^2} dx.
 \end{eqnarray*}
 By assuming $st>\frac{1}{4b^2}$,
\begin{equation}\label{eq:3.10}
\int_{\mathbb{R}} \lambda \,e^{\left(\frac{1}{4sb^2}-t \right)x^2} dx < \infty.    
\end{equation}
Recalling assumption  \eqref{4.4} and the fact \eqref{eq:3.10}, we conclude that
\begin{equation*}
 \int_{\mathbb{R}} ln^+(|h(x)|)\, dx < \infty, \quad \text{for}\,\,\,  st>\frac{1}{4b^2}. \end{equation*}

Since the function $h$ satisfies the hypothesis of  Lemma \ref{l:4.1}, $h$ must be a constant. Therefore, from \eqref{e:3.9}, we have

$$D_k^M(f)(z) = Ce^{-\frac{1}{4sb^2}z^2}e^{\frac{i}{2}\frac{d}{b}z^2}.$$ 
The assumptions  $st>\frac{1}{4b^2}$ and \eqref{4.4}, forces the constant $C$ to be zero. As a consequence of one-to-one of  $D_k^M$, we have  $f=0$.
 \item[$(ii).$] Let us consider $st=\frac{1}{4b^2}$. In the previous case, we proved that 
\begin{eqnarray}\label{e:3.10}
D_k^M(f)(x) &=& Ce^{-tx^2}\, e^{-\frac{i}{2}\frac{d}{b}x^2}.
\end{eqnarray} 
  If $|C|\le\lambda$, then from \eqref{e:3.10} we obtain
 \begin{eqnarray*}
\int_{\mathbb{R}}ln^+\left(\frac{|e^{tx^2}\, D_k^M(f)(x)|}{\lambda}\right)\, dx = \int_{\mathbb{R}}ln^+\left(\frac{|C|}{\lambda}\right)\, dx<\infty.
 \end{eqnarray*}
 Using the Property \ref{pr:2.2} (iii), we obtain that 
 \begin{eqnarray*}
     f(x) &=& {D_k^M}^{-1}\left(C\,e^{-\frac{(\cdot)^2}{4sb^2}}\,e^{\frac{i}{2}\frac{d}{b}(\cdot)^2}\right)(x)\\
     &=& \frac{C}{(-ib)^{k+1}}\int_{\mathbb{R}} e^{-\frac{y^2}{4sb^2}}\, e^{-\frac{i}{2}\frac{a}{b}x^2}\, E_k\left(\frac{iy}{b},x\right)\, d\mu_k(y)\\
     &=& C\, e^{-\left(\frac{i}{2}\frac{a}{b}+s\right)x^2}.
\end{eqnarray*}

\item[$(iii).$] We remain to prove the third part. Suppose $st<\frac{1}{4b^2}$. Let us consider $s<r<\frac{1}{4b^2t}$. Clearly, we observe that $$e^{(s-r-\frac{i}{2}\frac{a}{b})x^2} \in L^p_k(\mathbb{R}), \,\, \text{for} \,\,1\le p\le\infty.$$ Define $G_r(x) = e^{-(r+\frac{i}{2}\frac{a}{b})x^2}.$ The function   
$ e^{sx^2}G_r \in L^p_k(\mathbb{R})+L^{q}_k(\mathbb{R})$. By using Property \ref{pr:2.2} (iii), one can immediately obtain the identity $$D_k^M(G_r)(y) = C\, e^{-\frac{y^2}{4b^2r}}\, e^{\frac{i}{2}\frac{d}{b}y^2}.$$  Thus, for $r<\frac{1}{4b^2t}$, we have
\begin{eqnarray*}
\int_{\mathbb{R}} ln^+\left( \frac{|e^{tx^2} D_k^M(G_r)(x)|}{\lambda}\right) dx< \infty.
\end{eqnarray*}
This completes the proof.
 \end{itemize}
\end{proof}
\subsection{Cowling-Price uncertainty principle}
We start this subsection by recalling the essential results from the Dunkl transform. Following that, we will develop the Cowling-Price uncertainty principle within the framework of the LCDT, which is analogous to the Dunkl transform \cite{gallardo2004lp}.

\begin{lemma}\label{l:3.4}\cite{kawazoe2010uncertainty}
Let $\psi$ be a homogeneous polynomial in $\mathbb{R}$. Then there exists a homogeneous polynomial $Q$ with  $deg (Q) = deg(\psi)$ such that 
\begin{equation*}
D_k(\psi\,e^{-\delta\,|\cdot|^2})(x) = Q(x)\, e^{-\frac{|\cdot|^2}{4 \delta}}, \quad \text{ for all} \,\,\,\, \delta >0.    
\end{equation*}
\end{lemma}
\begin{lemma} \label{l:3.5}\cite{ray2004cowling}
Let $g$ be an entire function on $\mathbb{C}$ satisfying 
\begin{equation*}
|g(z)| \le C\, e^{s\,|Re\, z|^2}\,(1+|Im\, z|)^l  \text{ for some $l>0, s>0$},  
\end{equation*}
and
\begin{equation*}
\int_{\mathbb{R}}  \frac{|g(x)|^q}{(1+|x|)^m}\, |Q(x)|\, dx < \infty,  
\end{equation*}
for some $q\ge 1, \, m>1$ and a polynomial $Q$ of degree $M$. Then
$g$ is a polynomial with $deg (g) \le min \{l, \frac{m-M-1}{q}\}$ and if  $m<q+M+1$, then $g$ is a constant.
\end{lemma}
\begin{theorem}
Let $f$ be a measurable function on $\mathbb{R}$  such that
\begin{equation} \label{e:3.11}
\int_{\mathbb{R}} \frac{e^{spx^2}\,|f(x)|^p}{(1+|x|)^n}\, d\mu_k(x) <\infty 
\end{equation}
and 
\begin{equation} \label{e:3.12}
\int_{\mathbb{R}} \frac{e^{tq\lambda^2} |D_k^M(f)(\lambda)|^q}{(1+|\lambda|)^m}\, d\lambda <\infty,
\end{equation}
where $s,t>0$ and $n>0, m>1, 1\le p,q <\infty$. Then the following results hold.
\begin{itemize}
    \item [$(i)$] If $st>\frac{1}{4b^2}$, then $f=0$ a.e on $\mathbb{R}$.
    \item[$(ii)$] If $st= \frac{1}{4b^2}$, then  $f(x) = Q(x)\, e^{-(s+\frac{i}{2}\frac{a}{b})x^2}$, where $Q$ is a polynomial with $deg(Q) \le min \{\frac{n}{p}+ \frac{2k}{p'}, \frac{m-1}{q}\}$ and $p'$ is the conjugate exponent of $p$. Moreover, if  $1<m \le 1+q$ and $n>2k+1$, then $f(x)= C\,  e^{-(s+\frac{i}{2}\frac{a}{b})x^2}$. 
\item[$(iii)$] If $st<\frac{1}{4b^2}$, then there are many functions of the form $f(x) = Q(x)\,e^{-(\delta
+\frac{i}{2}\frac{a}{b})x^2}$ which satisfy the hypothesis, where $\delta \in (t, \frac{1}{4sb^2})$ and $Q$ is the polynomial in $\mathbb{R}$. 
\end{itemize}
\end{theorem}
\begin{proof}
We observe from \eqref{e:3.11}, $f$ is an  integrable function on $\mathbb{R}
$. Thus, $D_k^M(f)$ is well defined on $\mathbb{R}$. Note \ref{N:2.4} implies that  $D_k^M(f)$ is also an entire function on $\mathbb{C}$. 
Now, from \eqref{eq:4.2} and utilizing  H\"older's inequality along with Property \ref{pr:2.2} (ii), we have
\begin{eqnarray*}
|e^{-\frac{i}{2}\frac{d}{b}z^2}\,D_k^M(f)(z)| &\le& \frac{1}{|b|^{k+1}}\, \int_{\mathbb{R}} |f(x)|\, \left|E_k\left(-\frac{iz}{b},x\right)\right|\, d\mu_k(x)\\
&\le& \frac{1}{|b|^{k+1}}\, \int_{\mathbb{R}} |f(x)|\, e^{|\frac{Im z}{b}|\,|x|}\,d\mu_k(x)\\
&=& \frac{e^{\frac{|Im\,z|^2}{4sb^2}}}{|b|^{k+1}}\, \int_{\mathbb{R}} \frac{e^{s\,|x|^2}\,|f(x)|}{(1+|x|)^{\frac{n}{p}}}\, (1+|x|)^{\frac{n}{p}}\, e^{-s\,\left(|x|-\frac{Im\, z}{ 2sb}\right)^2}\,d\mu_k(x)\\
&\le& \frac{ C\,e^{\frac{|Im\,z|^2}{4sb^2}}}{|b|^{k+1}}\, \left(\int_{\mathbb{R}} (1+|x|)^{\frac{np'}{p}}\, e^{-sp'\,\left(|x|- \frac{Im\, z}{ 2sb}\right)^2}\,d\mu_k(x)\right)^{\frac{1}{p'}}\\ 
&\le& \frac{ C\,e^{\frac{|Im\,z|^2}{4sb^2}}}{|b|^{k+1}}\,\left(1+|Im\,z| \right)^{\frac{n}{p}+\frac{2k+1}{p'}}.
\end{eqnarray*} 
Now, we define the entire function 
\begin{equation} \label{e:3.13}
g(z) = e^{\frac{z^2}{4sb^2}}\, e^{-\frac{i}{2}\frac{d}{b}z^2}\, D_k^M(f)(z).
\end{equation} 
 Indeed, $g$ is pointwise bounded 
\begin{equation*}
|g(z)| \le   C\,e^{\frac{|Re\,z|^2}{4sb^2}}\,\left(1+|Im\,z| \right)^{\frac{n}{p}+\frac{2k+1}{p'}}.
\end{equation*}
Further, if $st= \frac{1}{4b^2}$, then assumption \eqref{e:3.12}, gives 
\begin{equation*}
    \int_{\mathbb{R}} \frac{ |g(x)|^q}{(1+|x|)^m}\,dx <\infty.  
\end{equation*}
From Lemma \ref{l:3.5}, we get $g$ is the polynomial  say $P(x)$ with $deg (P) \le r $, where $r=min \{\frac{n}{p}+\frac{2k+1}{p'}, \frac{m-1}{q}\}$. Thus, we deduce \eqref{e:3.13} 
\begin{eqnarray*}
D_k^M(f)(x) &=& P(x)\,e^{-\frac{x^2}{4sb^2}}\, e^{\frac{i}{2}\frac{d}{b}x^2} \\ 
&=&P(x)\,e^{-tx^2}\, e^{\frac{i}{2}\frac{d}{b}x^2}. 
\end{eqnarray*}
 By using relation  \eqref{eq:4.2} and Lemma \ref{l:3.4}, we obtain 
$f(x) = C\,Q(x)\, e^{-(s+\frac{i}{2}\frac{a}{b})x^2},$ 
where $Q$ is a polynomial with $deg(Q)=deg(P)$. Therefore, $f$ satisfies hypothesis \eqref{e:3.11}, whenever $n>2k+1+pr$. Consider  $m\le 1+q$. From Lemma \ref{l:3.5}, it follows that $g$ is a constant. Consequently, the equation \eqref{e:3.13} implies $$D_k^M(f)(x) = C\,e^{-\frac{x^2}{4sb^2}}\, e^{\frac{i}{2}\frac{d}{b}x^2}$$ and $f(x)= C\, e^{-(s+\frac{i}{2}\frac{a}{b})x^2}$. The choice of $n>2k+1$ and $m>1$, ensures that the hypothesis \eqref{e:3.11} and \eqref{e:3.12} satisfied, respectively. This completes the proof of $(ii)$.\\\\ $(i).$ If $st>\frac{1}{4b^2}$, then we choose $s', t'$ such that $s> s'=\frac{1}{4b^2t'}>\frac{1}{4b^2t}$. For $s'=\frac{1}{4b^2t'}$, (ii) implies that $f$ and $D_k^M(f)$  satisfy the hypothesis and $D_k^M(f)(x)= P(x)\,e^{-t'x^2}\, e^{\frac{i}{2}\frac{d}{b}x^2}$. But $D_k^M(f)$ cannot satisfy \eqref{e:3.11} unless $P(x)=0$. This follows that $f=0 \quad a.e$ on $\mathbb{R}$. \\
$(iii).$ If $st< \frac{1}{4b^2}$, repeating the same argument as in the previous case, we obtain $f(x) =  C\,Q(x)\, e^{-(\delta+\frac{i}{2}\frac{a}{b})x^2}$, for all $\delta \in (t, \frac{1}{4sb^2})$, which satisfies  \eqref{e:3.10} and \eqref{e:3.11}. This completes the proof.
\end{proof}
\subsection*{Acknowledgements:} The second author acknowledges the funding received from DST-SERB (SUR/2022/005678).
\subsection*{Data availability:} No new data was collected or generated during the course of this research.
\subsection*{Disclosure statement:} The authors report there are no competing interests to declare.  

\subsection*{Conflict of interest:} No potential conflict of interest was reported by the author.

\end{document}